\documentclass{amsart}
\usepackage{amssymb,xspace}
\usepackage{tikz-cd}
\usepackage{hyperref}
\input{commands.tex}

\begin{document}
\title[Higgs bundles over $\bP^1$]{Higgs bundles over $\bP^1$ and quiver representations}

\author{Sergey Mozgovoy}
\address{School of Mathematics, Trinity College Dublin, Ireland
\newline\indent
Hamilton Mathematics Institute, Ireland}
\email{mozgovoy@maths.tcd.ie}


\begin{abstract}
In this note we prove a new explicit formula for the invariants of moduli spaces of twisted Higgs bundles over $\bP^1$ and we relate these invariants to the invariants of moduli spaces of representations of some infinite symmetric quiver.
The formula can be easily implemented for computations and it is in agreement with the existing conjectures for the invariants of twisted Higgs bundles over curves.
Results of this note were presented in Oberwolfach in May 2014.
\end{abstract}

\maketitle


\section{Introduction}
Let $X=\bP^1$ be the projective line over a field \bk, $D$ be a divisor on $X$ and $L=\om_X(D)$ be the corresponding line bundle of degree $\ell\ge0$.
In this note we will study $L$-Higgs bundles $(E,\vi)$, where $E$ is a vector bundle over $X$ and $\vi:E\to E\ts L$ is a morphism.

Given coprime numbers $r,d$, let $M_L(r,d)$ be the moduli space of semistable $L$-Higgs bundles having rank $r$ and degree $d$.
One can ask about the motive of $M_L(r,d)$ if $\bk=\bC$ or about the number of points of $M_L(r,d)$ if $\bk=\bF_q$ is a finite field.
More generally, given a vector bundle $E$ over $\bP^1$, we can write it in the form $E=\bop_{i\in\bZ}O_X(i)^{\oplus m_i}$ for some map $m:\bZ\to\bN$ with finite support.
Let us define
\eq{M_L(m)
=\sets{\vi:E\to E\ts L}{(E,\vi)\text{ is semistable}}/\Aut E.}
Then we can write the motive (or the number of points)
\eq{\mot{M_L(r,d)}=\sum_{\ov{m:\bZ\to\bN}{r(m)=r,d(m)=d}}\mot{M_L(m)},}
where $r(m)=\sum_{i\in\bZ} m_i$ and $d(m)=\sum_{i\in\bZ}im_i$.
Generalizing the first question, one can ask about the number of elements in $M_L(m)$. In particular, one can conjecture that $M_L(m)$ is polynomial count (\cf \cite{letellier_higgs} in the case of parabolic Higgs bundles).

To answer the first question, let us consider an infinite symmetric quiver $Q$ (which depends on $\ell=\deg L$) with the set of vertices \bZ and the number of arrows from $i\in\bZ$ to $j\in\bZ$ equal
\eq{\max\set{\ell+1-\n{i-j},0}.}
We will interpret a map $m:\bZ\to\bN$ with finite support as a dimension vector of $Q$-representations.
Given a $Q$-representation $V$, we define its slope
\eq{\mu(V)=\frac{d(\udim V)}{r(\udim V)},\qquad \udim V=(\dim V_i)_{i\in\bZ},}
and we say that $V$ is semistable if for any subrepresentation $V'\sbs V$, we have $\mu(V')\le\mu(V)$.
Let $M_Q(m)$ denote the moduli space of semistable $Q$-representations having dimension vector $m$.
One of the main results of this paper is

\begin{theorem}
If $r$ and $d$ are coprime, then
\eq{\label{eq:th1}
\vmot{M_L(r,d)}
=\sum_{\ov{m:\bZ\to\bN}{r(m)=r,d(m)=d}}
\vmot{M_Q(m)},}
where
$\vmot Y=\bL^{-\oh\dim Y}\mot Y$ and $\bL=\mot{\bA^1}$.
Moreover, if $d>\ell\binom r2$, then $M_Q(m)$ is empty unless $m_i=0$ for $i\le0$.
\end{theorem}

This result is of combinatorial nature and I don't know if there is any relation between the corresponding categories of vector bundles and quiver representations.
Also, it is not quite clear if there is any geometric relation between $M_L(m)$ and $M_Q(m)$ although it is tempting to suggest that there is one.

Invariants on the right hand side of \eqref{eq:th1} can be easily computed \eqref{eq:omq}.
Therefore \eqref{eq:th1} gives an explicit formula for the motive of $M_L(r,d)$.
Let us now discuss a generalization of the previous theorem in the context of Donaldson-Thomas invariants.
Given $r\in\bZ_{>0}$ and $d\in\bZ$, one can define refined Donaldson-Thomas invariants $\Om_L(r,d)$ that count (with a weight) semistable $L$-Higgs bundles having rank $r$ and degree $d$ \eqref{eq:oml}.
In particular, if $r,d$ are coprime, then $\Om_L(r,d)=\vmot{M_L(r,d)}$.
Similarly, given a dimension vector $m:\bZ\to\bN$, one can define Donaldson-Thomas invariants $\Om_Q(m)$ that count semistable $Q$-representations having dimension vector $m$ \eqref{eq:omq}.

\begin{theorem}
For every $r\in\bZ_{>0}$, $d\in\bZ$, we have
\eq{\label{eq:th2}
{\Om_L(r,d)}
=\sum_{\ov{m:\bZ\to\bN}{r(m)=r,d(m)=d}}{\Om_Q(m)}.}
Moreover, if $d>\ell\binom r2$, then $\Om_Q(m)=0$ unless $m_i=0$ for $i\le0$.
\end{theorem}

Note that $\Om_L(r,d+r)=\Om_L(r,d)$ and we can always assume that $d>\ell\binom r2$.
The above sum contains only finitely many nonzero summands.
Invariants $\Om_Q(m)$ are polynomials in $\bL^{\pm\oh}$ \cite{efimov_cohomological}
and can be explicitly computed \eqref{eq:omq}.
Similarly to invariants $\Om_L(r,d)$, invariants $\Om_Q(m)$ are unchanged under shifts, that is, $\Om_Q(m)=\Om_Q(m[1])$, where $m[1]_i=m_{i-1}$ for $i\in\bZ$ (it satisfies $d(m[1])=d(m)+r(m)$ and $r(m[1])=r(m)$).
A more effective way to compute invariants $\Om_L(r,d)$ is the following

\begin{theorem}\label{thm:effective}
Define invariants $\Om_L^+(r,d)$ by the formula
\eq{
\sum_{r,d>0}\frac{\Om_L^+(r,d)}\lf z^rw^d
=\Log\rbb{\sum_{m:\bZ_{>0}\to\bN} \frac{\bL^{-\oh\hi(m,m)}}{\prod_i(\bL\inv)_{m_i}} z^{r(m)}w^{d(m)}}
,}
where $\Log$ is the plethystic logarithm (see \eg \cite{mozgovoy_computational}),
$\hi$ is the Euler-Ringel form of $Q$ \eqref{euler}
and $(q)_n=\prod_{i=1}^n(1-q^i)$
is the $q$-Pochhammer symbol. 
Then 
\eq{\Om_L(r,d)=\Om_L^+(r,d),\qquad d>\ell\binom r2.}
\end{theorem}



Let me now discuss relation of the above results to the existing literature.
A formula for the invariants $\Om_L(r,d)$ for arbitrary curves was conjectured in \cite{mozgovoy_solutions} based on the conjectures of \cite{chuang_motivic} and \cite{hausel_mixed}.
In particular, it was conjectured in \cite{chuang_motivic} that $\Om_L(r,d)$ are independent of $d$.
These conjectures are in agreement with the results \cite{hitchin_self-duality,gothen_betti,garcia-prada_motives} for low ranks and $D=0$.
In the case of $X=\bP^1$, $\deg D=4$ and low ranks, these conjectures were verified in \cite{rayan_co}.
An explicit formula (for arbitrary curves) was
proved in \cite{schiffmann_indecomposable} for $D=0$ and coprime $r,d$ and in \cite{mozgovoy_counting} for arbitrary $D$ and $r,d$.
This formula is rather complicated even for a computer implementation and it is different from the above conjectures.
The formulas of this paper work only for $\bP^1$, but they have an advantage of being easily computable although they are still different from the above conjectures.
As expected, all tests based on these formulas agree with the conjectures. These tests also show that $\Om_L(r,d)$ are independent of~$d$.

I would like to thank 
Emmanuel Letellier, 
Steven Rayan,
and Markus Reineke
for useful discussions.

\section{Higgs bundles over 
\tpdf{$\bP^1$}{P1}}

\subsection{Arbitrary curve}
Let $X$ be a curve of genus $g$, $D$ be a divisor on $X$ and $L=\om_X(D)$ be the corresponding line bundle of degree $\ell\ge\max\set{2g-2,0}$.
We define a (coherent) $L$-Higgs sheaf to be a pair $\h E=(E,\vi)$, where $E$ is a coherent sheaf and $\vi:E\to E\ts L$ is a morphism.
Let $\Higgs_L$ be the category of all $L$-Higgs sheaves and let $\Higgs^+_L$ be the category of $L$-Higgs sheaves $(E,\vi)$ such that
the factors in the Harder-Narasimhan filtration of $E$ (we call them HN-factors) have slopes $>0$ (we call such sheaves positive).

\begin{proposition}
For any $\h E,\h F\in\Higgs_L$, there is a long exact sequence
\begin{multline*}
0\to\Hom(\hE,\hF)\to\Hom(E,F)\to\Hom(E,F\ts L)\\
\to\Ext^1(\hE,\hF)\to\Ext^1(E,F)\to\Ext^1(E,F\ts L)
\to\Ext^2(\hE,\hF)\to0
\end{multline*}
\end{proposition}
\begin{proof}
See \cite{gothen_homological}.
\end{proof}

\begin{corollary}
The category $\Higgs_L$ has homological dimension two.
The Euler characteristic satisfies 
\[\hi(\h E,\h F)=-\ell\rk E\rk F.\]
In particular, $\hi(\h E,\h F)=\hi(\h F,\h E)$.
\end{corollary}

Let $\cM_L(r,d)$ (resp.\ $\cM_L^+(r,d)$) be the moduli stack of semistable $L$-Higgs bundles of type $(r,d)$ in $\Higgs_L$ (resp.\ in $\Higgs_L^+$).
Their expected dimension is $-\hi(\hE,\hE)=\ell r^2$, where $\hE\in\cM_L(r,d)$.
Define motivic invariants (if counting points over a finite field $\bF_q$, we use $\bL^\oh=-q^\oh$)
\eq{\cI_L(r,d)=\bL^{-\oh\ell r^2}\mot{\cM_L(r,d)},}
where $\mot{X/\GL_n}:=\mot{X}/\mot{\GL_n}$ for a global stack $X/\GL_n$.
Define motivic DT invariants (see \eg \cite{mozgovoy_computational} for the definition of plethystic operations)
\eq{\label{eq:oml}
\sum_{d/r=\ta}\frac{\Om_L(r,d)}\lf t^r
=\Log\rbr{1+\sum_{d/r=\ta} \cI_L(r,d)t^r},\qquad \ta\in\bQ.}
Note that if $r,d$ are coprime then \eq{\Om_L(r,d)=\lfb\cI_L(r,d)=
\bL^{-\oh\dim M_L(r,d)}\mot{M_L(r,d)},}
where $M_L(r,d)$ is the moduli space of semistable Higgs bundles having rank $r$ and degree $d$.
Similarly, we define invariants $\cI^+_L(r,d)$ and $\Om^+_L(r,d)$ corresponding to the stacks $\cM_L^+(r,d)$.

\begin{remark}
Computations suggest that $\Om_L(r,d)$ are independent of $d$.
Note that $\cI_L(r,d+r)=\cI_L(r,d)$ and therefore
$\Om_L(r,d+r)=\Om_L(r,d)$.
The conjectural formula for $\cI_L(r,d)$ is given in \cite{mozgovoy_solutions}.
Invariants $\Om_L^+(r,d)$ can depend on $d$ but they stabilize for large $d$.
\end{remark}

\begin{proposition}
\label{prp:+vs all}
If $d>\ell\binom r2$ then
\begin{enumerate}
\item $\cM_L^+(r,d)=\cM_L(r,d)$,
\item $\cI_L^+(r,d)=\cI_L(r,d)$,
\item $\Om_L^+(r,d)=\Om_L(r,d)$.
\end{enumerate}
\end{proposition}
\begin{proof}
1) is proved in \cite{mozgovoy_counting}.
2) follows from 1).
To prove 3), we note that if $d'/r'=d/r$ and $r'\le r$, then
$d'/r'=d/r>\ell\frac{r-1}2\ge\ell\frac{r'-1}2$
and therefore $d'>\ell\binom{r'}2$.
This implies that $\cI^+_L(r',d')=\cI_L(r',d')$.
Now we apply equation \eqref{eq:oml} and a similar equation for $\Om_L^+(r,d)$ and conclude that $\Om_L(r,d)=\Om_L^+(r,d)$.
\end{proof}

As $\cI_L(r,d+r)=\cI_L(r,d)$, we can determine invariants $\cI_L(r,d)$ if we know $\cI_L^+(r,d)$ for $d\gg0$.
Let $\cJ_L^+(r,d)$ denote the weighted count of all positive Higgs bundles (with the twist $\bL^{-\oh\ell r^2}$ as before), that is,
\eq{\cJ_L^+(r,d)=\bL^{-\oh\ell r^2}\sum_{\ov{\cha E=(r,d)}{E-\text{pos.}}}\frac{\mot{\Hom(E,E\ts L)}}{\mot{\Aut E}}.}

\begin{proposition}\label{prp:jl+ vs il+}
If $\ell\ge 2g-2$, then
\begin{gather}
1+\sum_{r,d>0}\cJ_L^+(r,d)z^rw^d
=\prod_{\ta>0}\rbb{1+\sum_{d/r=\ta}\cI_L^+(r,d)z^rw^d},\\
\label{eq:omp1}
\Log\rbb{1+\sum_{r,d>0} \cJ^+_L(r,d)z^rw^d}
=\sum_{r,d>0}\frac{\Om_L^+(r,d)}\lf z^rw^d.
\end{gather}
\end{proposition}
\begin{proof}
The first equation is the consequence of the existence of Harder-Narasimhan filtrations in the category $\Higgs^+_L$.
The second equation follows from the first one and the definition of DT invariants.
\end{proof}

Our goal is to determine invariants $\cJ^+_L(r,d)$ (for the projective line), and then to compute DT-invariants $\Om_L(r,d)$ using the sequence:
\[\cJ^+_L(r,d)\to\Om^+_L(r,d)\to \Om_L(r,d).\]

\subsection{Projective line}
We will give an explicit formula for the invariants $\cJ^+_L(r,d)$ in the case of the curve $\bP^1$.
As before, for any $m:\bZ\to\bN$ with finite support, define
\eq{r(m)=\sum_{i\in\bZ} m_i,\qquad d(m)=\sum_{i\in\bZ}im_i.}
We can write the Euler-Ringel form for the quiver $Q$ defined in the introduction as
\eq{\label{euler}
\hi(m,m')=\sum_i m_im'_i-\sum_{\n{i-j}\le \ell}m_im_j(\ell+1-\n{i-j}),}
where $m,m':\bZ\to\bN$ have finite support.
Finally, define the $q$-Pochhammer symbol
\eq{(q)_m=\prod_{i\in\bZ}(q)_{m_i},\qquad (q)_n=(q;q)_n=\prod_{i=1}^n(1-q^i).}

\begin{theorem}\label{th:jl+}
We have
\eq{\label{eq:jl+}
\cJ^+_L(r,d)
=\sum_{\ov{m:\bZ_{>0}\to\bN}{r(m)=r,d(m)=d}}
\frac{\bL^{-\oh\hi(m,m)}}{(\bL\inv)_m}.
}
\end{theorem}
\begin{proof}
Given a vector bundle $E$ over $X=\bP^1$, we can write it in the form $E=\bop_{i\in\bZ}O_X(i)^{\oplus m_i}$, where $m:\bZ\to\bN$ has a finite support.
Moreover,
\[\rk E=r(m),\qquad \deg E=d(m).\]
The vector bundle $E$ is positive (\ie its Harder-Narasimhan factors have slopes $>0$) if and only if $m_i=0$ for $i\le0$.
Therefore we can consider $m$ as a map $\bZ_{>0}\to\bN$.
The contribution of $E$ to $\cJ^+_L(r,d)$ equals \cite{mozgovoy_motivicb}
\[\bL^{-\oh\ell r^2}\frac{\mot{\Hom(E,E\ts L}}{\mot{\Aut E}}
=\bL^{-\oh \ell r^2}\frac{\mot{\Hom(E,E\ts L}}{\mot{\End E}\cdot(\bL\inv)_m}.\]
The power of $\bL^\oh$ here is given by (we use $h^i(E,F):=\dim\Ext^i(E,F)$)
\begin{multline*}
-\ell r^2+2h^0(E,E\ts L)-2h^0(E,E)\\
=\sum_{i,j}(-\ell+2[i-j+\ell+1]_+-2[i-j+1]_+)m_im_j,
\end{multline*}
where $[n]_+=\max\set{0,n}$.
Consider the sum of contributions corresponding to pairs $(i,j)$ and $(j,i)$ with $i>j$.
If $0<i-j\le \ell$, then the sum is
$$-2\ell+4(\ell+1)-2(i-j+1)=2(\ell+1-\n{i-j}).$$
If $i-j>\ell$, then the sum is
\[-2\ell+2\ell=0.\]
If $i=j$, the contribution is
\[-\ell+2\ell=(\ell+1-\n{i-j})-1.\]
These computations imply that the power of $\bL^\oh$ is equal to
\[\sum_{\n{i-j}\le \ell}m_im_j(\ell+1-\n{i-j})-\sum_im_i^2=-\hi(m,m)\]
and the contribution of $E$ to $\cJ^+_L(r,d)$ is
$\bL^{-\oh\hi(m,m)}/(\bL\inv)_m$.
Finally, we take the sum over all $m:\bZ_{>0}\to\bN$ that correspond to positive vector bundles of rank $r$ and degree $d$.
\end{proof}

\begin{remark}
This theorem gives an effective way to compute invariants $\Om_L(r,d)$:
\begin{enumerate}
\item Compute invariants $\cJ_L^+(r,d)$ using \eqref{eq:jl+}.
\item Compute invariants $\Om_L^+(r,d)$ using \eqref{eq:omp1}.
\item Use equation $\Om_L(r,d)=\Om_L^+(r,d)$ for $d>\ell\binom r2$.
\end{enumerate}
\end{remark}


\section{Relation to quivers}
Let $Q$ be the infinite quiver defined in the introduction.
For any dimension vector $m:\bZ\to\bN$, let $\cJ_Q(m)$ denote the weighted count of all representations having dimension vector $m$.
More precisely, we consider the stack of all quiver representations $\ma_Q(m)=R_Q(m)/\GL_m$, where
\eq{R_Q(m)=\bop_{(a:i\to j)\in Q_1}\Hom(\bk^{m_i},\bk^{m_j}),\qquad
\GL_m=\prod_{i\in\bZ}\GL_{m_i}.}
Then $\dim\ma_Q(m)=-\hi(m,m)$ and \cite{mozgovoy_motivic}
\eq{\label{eq:jq}
\cJ_Q(m)
=\bL^{-\oh\dim\ma_Q(m)}[\ma_Q(m)]
=\frac{\bL^{-\oh\hi(m,m)}}{(\bL\inv)_m}.
}
We define Donaldson-Thomas invariants $\Om_Q(m)$ by the formula
\eq{\label{eq:omq}
\sum_{m:\bZ\to\bN}\frac{\Om_Q(m)}\lf z^m=\Log\rbr{\sum_{m:\bZ\to\bN}\cJ_Q(m)z^m}.}

Similarly, consider the stack $\cM_Q(m)=R_Q^\sst(m)/\GL_m$ of semistable quiver representations, where $R_Q^\sst(m)\sbs R_Q(m)$ is the subspace of semistable representations,
and define
\eq{\cI_Q(m)
=\bL^{-\oh\dim\cM_Q(m)}[\cM_Q(m)].
}
\begin{proposition}
We have
\begin{gather}
\sum_{m:\bZ\to\bN}\cJ_Q(m)z^{m}
=\prod_{\ta\in\bQ}\rbb{1+\sum_{\ov{m:\bZ\to\bN}{d(m)/r(m)=\ta}}\cI_Q(m)z^{m}},\\
\sum_{\ov{m:\bZ\to\bN}{d(m)/r(m)=\ta}}\frac{\Om_Q(m)}\lf z^{m}
=\Log\rbb{1+\sum_{\ov{m:\bZ\to\bN}{d(m)/r(m)=\ta}} \cI_Q(m)z^{m}}.
\end{gather}
\end{proposition}
\begin{proof}
The proof is similar to the proof of Proposition \ref{prp:jl+ vs il+}.
\end{proof}

\begin{lemma}\label{lmm:d>>0}
Let $m:\bZ\to\bN$ be a dimension vector such that $d(m)>\ell\binom{r(m)}2$.
If there exist semistable representations of dimension vector $m$, then $m_i=0$ for $i\le0$.
\end{lemma}
\begin{proof}
Let $V$ be a semistable $Q$-representation of dimension vector $m$ and let
$\supp V=\sets{i\in\bZ}{V_i\ne0}$ be its support.
Let $[a,b]$ be an interval such that $a,b\in \supp V$, but $(a,b)\cap\supp V=\es$ (we call it a gap).
Then it has length $\le\ell$ as otherwise there are no arrows between $a,b$ in $Q$ and $V$ can be decomposed as a direct sum of representations having support in $(-\infty,a]$ and $[b,+\infty)$, which would imply that $V$ is not semistable.
Let $a_1<\dots<a_k$ be the elements of $\supp V$ and let $r_i=m_{a_i}>0$ be the corresponding dimensions (so that $r=r(m)=\sum_i r_i$).
Then $a_i-a_{i-1}\le\ell$ and therefore $a_i\le a_1+(i-1)\ell$.
This implies
$$\frac{r-1}2\ell<\mu(V)=\frac{\sum_{i=1}^k a_ir_i}r
\le\frac{\sum_{i=1}^k(a_1+(i-1)\ell)r_i}r
\le \frac kra_1+\frac{r-1}2\ell
$$
and therefore $a_1>0$.
This implies that $m_i=0$ for $i\le0$.
We used here the fact that if $\sum_i r_i=r$, then $\sum_i(i-1)r_i\le\binom r2$.
Indeed, by induction
$$\sum_i (i-1)r_i
\le\max_{1\le r_1\le r}\rbr{\binom{r-r_1}2+(r-r_1)}
=\max_{1\le r_1\le r}\binom{r-r_1+1}2
=\binom r2.
$$

\end{proof}

\begin{theorem}
We have
\eq{\Om_L(r,d)
=\sum_{\ov{m:\bZ\to\bN}{r(m)=r,d(m)=d}}\Om_Q(m).}
\end{theorem}
\begin{proof}
We obtain from Theorem \ref{th:jl+} and equation \eqref{eq:jq}
\eq{\label{eq:jl-q}
\cJ^+_L(r,d)
=\sum_{\ov{m:\bZ_{>0}\to\bN}{r(m)=r,d(m)=d}}\cJ_Q(m).}
Let us define DT invariants $\Om^+_Q(m)$ of the subquiver $Q^+$ of $Q$ with vertices $\bZ_{>0}$ by the formula
\eq{\sum_{m:\bZ_{>0}\to\bN}\Om^+_Q(m)z^m
=\lfb\Log\rbb{\sum_{m:\bZ_{>0}\to\bN}\cJ_Q(m)z^m}.}
Then we obtain from \eqref{eq:omp1} and \eqref{eq:jl-q}
\eq{
\Om_L^+(r,d)=\sum_{\ov{m:\bZ_{>0}\to\bN}{r(m)=r,d(m)=d}}\Om^+_Q(m).
}
Our goal is to extend this equation to arbitrary $L$-Higgs bundles and $Q$-represen\-ta\-tions. 
Invariants $\Om^+_Q(m)$ can be also defined by the formula
(\cf Proposition \ref{prp:jl+ vs il+})
\eq{\sum_{\ov{m:\bZ\to\bN}{d(m)/r(m)=\ta}}
\Om^+_Q(m)z^{m}
=\lfb\Log\rbb{1+\sum_{\ov{m:\bZ\to\bN}{d(m)/r(m)=\ta}} \cI^+_Q(m)z^{m}},}
where $\cI^+_Q(m)=\cI_Q(m)$ if $m_i=0$ for $i\le0$ and $\cI^+_Q(m)=0$ otherwise (it counts semistable representations of $Q^+$).
We obtain from Lemma \ref{lmm:d>>0} that $\cI^+_Q(m)=\cI_Q(m)$ if $d(m)>\ell\binom{r(m)}2$.
This implies, similarly to Proposition \ref{prp:+vs all}, that $\Om^+_Q(m)=\Om_Q(m)$ if $d(m)>\ell\binom{r(m)}2$.
Finally, we obtain for $d>\ell\binom r2$,
\eq{
\Om_L(r,d)
=\Om_L^+(r,d)
=\sum_{\ov{m:\bZ_{>0}\to\bN}{r(m)=r,d(m)=d}}\Om^+_Q(m)
=\sum_{\ov{m:\bZ\to\bN}{r(m)=r,d(m)=d}}\Om_Q(m).
}
Using shifts we can see that the same formula is true for all $r,d$.
\end{proof}

\begin{corollary}
$\Om_L(r,d)$ is a polynomial in $\bL^{\pm\oh}$.
\end{corollary}
\begin{proof}
Invariants $\Om_Q(m)$ are polynomials in $\bL^{\pm\oh}$ be the result of Efimov \cite{efimov_cohomological} about DT invariants of symmetric quivers.
\end{proof}

\section{Examples}
As I mentioned earlier, all tests show that invariants $\Om_L(r,d)$ are independent of $d$.
Let us denote them by $\Om_\ell(r)$, where $\ell=\deg L$.
To simplify notation, we will use a variable $w=\bL^\oh$.
Note that if $r,d$ are coprime, then the moduli space $Y=M_L(r,d)$ is smooth of dimension $\ell r^2+1$ and
\eq{\Om_L(r,d)=\vmot Y=w^{-\dim Y}\sum_{n\ge0}\dim H_c^n(Y,\bQ)w^n.}
By the Poincar\'e duality this polynomial has degree $\ell r^2+1$.
In order to compute it we determine $\Om^+_L(r,d)$ for $d>\ell\binom r2$ using Theorem \ref{thm:effective}.

\subsection{\tpdf{$\ell=1$}{l=1}}
We have
\begin{flalign*}
\Om_1(1)&=w^2&\\
\Om_1(2)&=w^5\\
\Om_1(3)&={w}^{10}+{w}^{8}\\
\Om_1(4)&
={w}^{17}+{w}^{15}+3\,{w}^{13}+2\,{w}^{11}\\
\Om_1(5)&
={w}^{26}+{w}^{24}+3\,{w}^{22}+5\,{w}^{20}+7\,{w}^{18}+9\,{w}^{16}+5\,{w}^{14}\\
\Om_1(6)&=
{w}^{37}+{w}^{35}+3\,{w}^{33}+5\,{w}^{31}+10\,{w}^{29}+13\,{w}^{27}+22
\,{w}^{25}
+27\,{w}^{23}\\
&+32\,{w}^{21}+29\,{w}^{19}+13\,{w}^{17}
\end{flalign*}

\subsection{\tpdf{$\ell=2$}{l=2}}
We have
\begin{flalign*}
\Om_2(1)&=w^3&\\
\Om_2(2)&=
{w}^{9}+{w}^{7}\\
\Om_2(3)&=
{w}^{19}+{w}^{17}+3\,{w}^{15}+4\,{w}^{13}+3\,{w}^{11}\\
\Om_2(4)&=
{w}^{33}+{w}^{31}+3\,{w}^{29}+5\,{w}^{27}+9\,{w}^{25}+13\,{w}^{23}+18
\,{w}^{21}+22\,{w}^{19}+20\,{w}^{17}\\
&+10\,{w}^{15}\\
\Om_2(5)&=
{w}^{51}+{w}^{49}+3\,{w}^{47}+5\,{w}^{45}+10\,{w}^{43}+15\,{w}^{41}+26
\,{w}^{39}+38\,{w}^{37}
+56\,{w}^{35}\\
&+77\,{w}^{33}
+105\,{w}^{31}
+131\,{
w}^{29}+156\,{w}^{27}+165\,{w}^{25}
+154\,{w}^{23}+103\,{w}^{21}\\
&+40\,{w}^{19}
\end{flalign*}
These formulas coincide with the results of \cite{rayan_co} for co-Higgs bundles.
\begin{flalign*}
\Om_2(6)&=
{w}^{73}+{w}^{71}+3\,{w}^{69}+5\,{w}^{67}+10\,{w}^{65}+16\,{w}^{63}+28
\,{w}^{61}
+42\,{w}^{59}+68\,{w}^{57}&\\
&+100\,{w}^{55}+147\,{w}^{53}+207\,{w}^{51}+292\,{w}^{49}+392\,{w}^{47}+524\,{w}^{45}+678\,{w}^{43}\\
&+858\,{w}^{41}+1050\,{w}^{39}+1253\,{w}^{37}+1427\,{w}^{35}+1537\,{w}^{33}+1531\,{w}^{31}\\
&+1364\,{w}^{29}+1022\,{w}^{27}+557\,{w}^{25}+171\,{w}^{23}
\end{flalign*}

\subsection{\tpdf{$\ell=3$}{l=3}}
We have
\begin{flalign*}
\Om_3(1)&={w}^{4}&\\
\Om_3(2)&=
{w}^{13}+{w}^{11}+2\,{w}^{9}\\
\Om_3(3)&=
{w}^{28}+{w}^{26}+3\,{w}^{24}+4\,{w}^{22}+7\,{w}^{20}+9\,{w}^{18}+9\,{
w}^{16}+6\,{w}^{14}\\
\Om_3(4)&=
{w}^{49}+{w}^{47}+3\,{w}^{45}+5\,{w}^{43}+9\,{w}^{41}+13\,{w}^{39}+22
\,{w}^{37}+30\,{w}^{35}
+45\,{w}^{33}\\&
+56\,{w}^{31}+75\,{w}^{29}+85\,{w}
^{27}+97\,{w}^{25}+87\,{w}^{23}+63\,{w}^{21}+28\,{w}^{19}\\
\Om_3(5)&=
{w}^{76}+{w}^{74}+3\,{w}^{72}+5\,{w}^{70}+10\,{w}^{68}+15\,{w}^{66}+26\,{w}^{64}+38\,{w}^{62}+60\,{w}^{60}\\&
+85\,{w}^{58}+125\,{w}^{56}+172\,{w}^{54}+238\,{w}^{52}+315\,{w}^{50}+417\,{w}^{48}+529\,{w}^{46}\\&
+669\,{w}^{44}+819\,{w}^{42}+979\,{w}^{40}+1130\,{w}^{38}+1247\,{w}^{36}+1314\,{w}^{34}\\&
+1274\,{w}^{32}+1120\,{w}^{30}+816\,{w}^{28}+457\,{w}^{26}+155\,{w}^{24}
\end{flalign*}

All these computations are in agreement with conjectures of \cite{mozgovoy_solutions}.



\providecommand{\bysame}{\leavevmode\hbox to3em{\hrulefill}\thinspace}
\providecommand{\href}[2]{#2}

\end{document}